\documentclass[a4paper, 12pt]{article}
\usepackage[utf8]{inputenc}
\usepackage[T2A]{fontenc}
\usepackage{amsmath,amssymb,amsthm}
\usepackage[a4paper,hmargin=2.5cm,vmargin=2.5cm]{geometry}
\usepackage[english]{babel}
\usepackage{graphicx}
\usepackage{forest}
\usepackage{tikz-qtree}
\usepackage{algorithm}
\usepackage{algpseudocode}
\usepackage{tikz-cd}

\DeclareMathOperator{\Ker}{Ker }

\newtheorem{theorem}{Theorem}[section]
\newtheorem{lemma}[theorem]{Lemma}
\newtheorem{prop}[theorem]{Proposition}

\newtheorem{defn}[theorem]{Definition}

\newtheorem{corr}[theorem]{Corollary}

\begin{document}

\begin{center}
{\Large\bfseries
Bases for counting functions \\ on free monoids and groups
\par}

\vspace{0.7cm}

\textbf{\large Petr Kiyashko}
\end{center}

\begin{abstract}
   The focus of this paper is providing
   a description of the spaces of counting functions on
   free monoids and groups and of the Brooks space.
   These results have been originally obtained in \cite{ggd}, however
   one of the theorems at the time of writing this paper contained a
   slight error, and we correct it here.
   We also propose alternative, simpler proofs to the main theorems from \cite{ggd}.
\end{abstract}

\section{Introduction}
\paragraph{}
Let $S = \{a_1, a_2, \dots a_n\}$ 
be a finite set of letters with $n \geq 2$. 
We define the free monoid of rank $n$ as the set of all finite 
words over the alphabet $S$
(including the empty word, denoted by 
$\epsilon$), and denote it by $M_n$. We denote by
$|w|$ the length of the word 
$w \in M_n$, and by $w_i$ the $i$-th 
letter of the word, 
with $1$ being the first index (sometimes the $i$-th word of a set of words
may also be denoted by $w_i$, but it is 
always clear from the context which meaning we are referring to).
We denote by $w_{fin}$ the last letter of $w$ (i. e. $w_{|w|}$) 
and by $w_{fin-i}$ we
denote the letter $w_{|w|-i}$.
A word $v$ is a subword of another word $w$ if for some letter index
$i \leq |w| - |v| + 1$ it holds that 
$v = w_i w_{i+1}\ldots w_{i + |v| - 1}$. For any word $w \in M_n$
and $k \geq 0$
we define the $k$-th power of $w$ as the word obtained by concatenating $k$ copies of $w$ and denote it by $w^k$.

\paragraph{}
Similarly, let $\bar{S} = \{a_i, a_i^{-1} \mid a_i \in S\}$,
and we define
the free group of rank $n$ as the set of all reduced words over the 
alphabet $\bar{S}$, and denote it by $F_n$.
A word $w$ is called reduced if for all 
$1 \leq i \leq n$ the words $a_ia_i^{-1}$ and $a_i^{-1}a_i$ are not
subwords of $w$. The empty word is also included in $F_n$
by definition. The length, the $i$-th letter of a reduced word,
the subword relation between reduced words in $F_n$
are defined similarly to the monoid case. For the
case of a reduced word $w \in F_n$ with $w_1 \neq w_{fin}^{-1}$
the $k$-th power of $w$ denoted by $w^k$ is also
defined similarly to the monoid case. We define the inverse of
a reduced word $w$ as the word
$w^{-1} =w_{fin}^{-1}w_{fin-1}^{-1}\ldots w_1^{-1}$.

\paragraph{}
Let us define the function $p_v(w)$ for $v, w \in M_n$ as the 
number of occurrences of $v$ in $w$, i. e. the number of distinct indices
$i$ such that $v = w_i w_{i+1}\ldots w_{i + |v| - 1}$. This 
function will be referred to as the 
elementary $v$-counting function. 
For the free group $F_n$ the definition of elementary counting functions is 
virtually identical, with an additional requirement that the words 
$v$ and $w$ are necessarily reduced. We extend the definition
to the empty word as $p_{\epsilon}(w) = |w|$ in both cases.

Now, in the monoid case
we define a counting function $f$ over a field
$\Re \in \{\mathbb{Q}, \mathbb{R}, \mathbb{C}\}$ 
as an
arbitrary linear combination of elementary counting functions
with coefficients from $\Re$, i. e.
\begin{equation} \label{eq:cdef}
f = \sum_{i=1}^{m}x_ip_{w_i}, \ x_i \in \Re \setminus \{0\}, \ w_i \in M_n.
\end{equation}
The functions $p_w$ map $M_n$ to $\mathbb{N}$,
which in turn can be naturally embedded in $\Re$, thus if we consider the
elementary counting functions to be maps from
$M_n$ to $\Re$, then such linear combinations are valid. 
Let us denote by $C(M_n)$ the vector space of all 
counting functions. The counting functions for $F_n$
and the space $C(F_n)$ are defined similarly.
For a representation  of some counting function $f$ of the form (\ref{eq:cdef}) we define
its depth as the maximum of $|w_i|$ over all words $w_i$.
Now, for any function $f \in C(M_n)$ ($f \in C(F_n)$) 
we define the depth of $f$ as the minimal depth over all the representations of $f$.
It will be proven further that if the depth of $f$ is non-zero it has only one 
representation of the form (\ref{eq:cdef}).

\paragraph{}
We consider two counting functions $f_1, f_2$ equivalent 
($f_1 \sim f_2$) if their
difference is bounded, i. e. $||f_1 - f_2||_{\infty} < \infty$.
Note that this relation is linear, thus we may define a quotient space.
We denote by
$\widehat{f}$ the equivalency class of $f \in C(M_n)$ ($f \in C(F_n)$).
We denote by $K(M_n)$ (respectively $K(F_n)$) the set of all functions $f$ such that
$\widehat{f} = \widehat{0}$, i. e.
the kernel of factorization by the relation $\sim$, and by $\widehat{C}(M_n)$ 
(respectively $\widehat{C}(F_n)$) we denote the vector space of classes of functions 
$\widehat{f}$ for all $f \in C(M_n)$ ($f \in C(F_n)$),
i. e. the image of factorization by the relation $\sim$.

\paragraph{}
Finally, let us define the Brooks space.
A Brooks quasimorphism is defined as $$\phi_w = p_w - p_{w^{-1}}$$
for some reduced word $w \in F_n$ with $|w| \geq 1$. We define the 
Brooks space $Br(F_n)$ as the space of all linear combinations 
of Brooks quasimorphisms with coefficients from $\Re$. Two elements
of the Brooks space are considered equivalent ($\sim$) if their difference is bounded.
The quotient space $\widehat{Br}(F_n)$ is defined similarly to the monoid and 
group case.
Now consider the following linear mappings:
$\sigma_1: C(F_n) \rightarrow Br(F_n)$,
$\sigma_1(p_w) = \phi_w$ and $\sigma_2: Br(F_n) \rightarrow \widehat{Br}(F_n)$,
$\sigma_2(\phi_w) = \widehat{\phi_w}$,
and their composition $\sigma = \sigma_2 \circ \sigma_1$,
$\sigma: C(F_n) \rightarrow \widehat{Br}(F_n)$. 
\begin{figure}[h!]
  \centering
        \begin{tikzcd}
        C(F_n) \arrow[rd, "\sigma"] \arrow[r, "\sigma_1"] \arrow[d] &
        Br(F_n) \arrow[d, "\sigma_2"]\\
        \widehat{C}(F_n) \arrow[r] & \widehat{Br}(F_n)
        \end{tikzcd}
    \caption{Commutative diagram for $\sigma_1$, $\sigma_2$ and $\sigma$}
    \label{fig:comm}
\end{figure}
\\
We define the kernel space $K_{Br} (F_n)$ as the kernel of $\sigma$. Note that
\begin{equation} \label{eq:kern}
K_{Br} = \Ker \sigma = (\Ker \sigma_1) + \sigma_1^{-1}(\Ker \sigma_2).
\end{equation}

\paragraph{}
In this paper we sometimes refer to $C(M_n)$ and $C(F_n)$ as $C$,
to $\widehat{C}(M_n)$ and $\widehat{C}(F_n)$ as $\widehat{C}$ and
to $K(M_n)$ and $K(F_n)$ as $K$.
Whether we are referring to the monoid case, the group case or
both simultaneously should be clear from context.
We also refer to $Br(F_n)$, 
$\widehat{Br}(F_n)$ and $K_{Br}(F_n)$
as $Br$, $\widehat{Br}$ and $K_{Br}$ respectively.

\paragraph{}
The aim of this paper is describing spanning sets of
$K(M_n), K(F_n)$ and $K_{Br}$, identifying bases of 
$\widehat{C}(M_n), \widehat{C}(F_n)$ and $\widehat{Br}$, and
providing simpler proofs of these descriptions.

\subsection{Motivations}
\paragraph{}
The descriptions of the kernels and the bases were first laid out in
\cite{ggd}, and the Brooks space was first introduced in \cite{brooks}.
A description of its basis may also be found in \cite{ggd}, and
since this basis is infinite, this leads to the fact
that dim $H_b^2(F_n;\mathbb{R}) = \infty$, which Brooks tried to
prove, and which has been proven in \cite{mits} by constructing 
an infinite set of linearly independent classes of Brooks 
quasimorphisms. Larger sets have been since 
provided in \cite{fiziev}
and \cite{grig}. However, these sets were not
bases of $\widehat{Br}(F_n)$, and an explicit basis has been first
identified in \cite{ggd}. The proofs given in this paper
were relatively complicated and lengthy, therefore we tried to find shorter
and simpler alternatives. Moreover, the description of the basis of $\widehat{C}(F_n)$
in \cite{ggd} at the moment of writing this paper contained one extra
element, which rendered the provided set linearly dependent, and we provide an
updated description.
Alas, while quasimorphisms
on free groups have been extensively researched, the same can
not be said about quasimorphisms on free monoids, thus
the aforementioned 
descriptions in the case of $M_n$ may prove useful in
future research.

\paragraph{}
The descriptions of kernels $K$ and bases of $\widehat{C}$
are also used in algorithms that check
whether two counting functions are equivalent. 
One pen-and-paper algorithm for 
this problem has been proposed in \cite{ggd}, and a formal algorithm that is efficient for
free groups of rank $n \geq 2$ and free monoids of rank $n \geq 3$ has been described in \cite{final}.
Such algorithms may be used to explore the action of Out($F_n$)
on $H_b^2(F_n;\mathbb{R})$, namely fixed points, as
has been done in \cite{hase}. Similar applications may be found
in the case of quasimorphisms on $M_n$, 
though this case has not been explored yet.

\subsection{Acknowledgements}
I would like to thank my advisor Alexey Talambutsa
for the assistance in writing this work. I would also like to thank Tobias Hartnick
and Antonius Hase for their reviews and valuable corrections that helped polish the paper. This work was partially supported by the BASIS foundation under grant №22-7-2-32-1.

\subsection{Main results}
\subsubsection{Kernel and basis}
\begin{defn}[from theorem 1.3, \cite{ggd}]
For a word $w \in M_n$ we define the left and right extension
relations as counting functions $l_w$ and $r_w$
of the following form:
$$l_w = p_w - \sum_{s \in S}p_{sw}, \
r_w = p_w - \sum_{s \in S}p_{ws}.$$
\end{defn}

\begin{defn}[from theorem 1.3, \cite{ggd}]
Similarly, for a reduced word $w \in F_n$ we
define the left and right extension relations as follows:
$$l_w = p_w - \sum_{s \in \bar{S} \setminus \{w_1^{-1}\}}p_{sw}, \
r_w = p_w - \sum_{s \in \bar{S} \setminus \{w_{fin}^{-1}\}}p_{ws}.$$
\end{defn}

\begin{defn}[from theorem 1.3, \cite{ggd}, refined]
For a reduced word $w \in F_n$ we define the symmetry relation
$s_w = p_w + p_{w^{-1}}$. We also define the symmetrized
extension relation $se_w = l_w - r_{w^{-1}}$.
\end{defn}

\begin{theorem}[Kernel, theorem 1.3, \cite{ggd}] \label{teor:ker}
    : \\
    1) $K(M_n)$ is spanned by $\{ l_w, r_w \mid w \in M_n\}$. \\
    2) $K(F_n)$ is spanned by $\{ l_w, r_w \mid w \in F_n\}$. \\
    3) $K_{Br}(F_n)$ is spanned by $\{ s_w, se_w \mid w \in F_n\}$.
\end{theorem}

\begin{theorem}[Basis, theorem 1.5, \cite{ggd}, corrected] \label{teor:bas}
    : \\
    1) Let $W$ be the set of all words $w \in M_n$, including $\epsilon$, such that 
    $w_1 \neq a_1$ and $w_{fin} \neq a_1.$ Then the set 
    $B = \{\widehat{p_w} \mid w \in W\}$
    is a basis of $\widehat{C}(M_n)$.\\ \\ 
    2) Similarly, let $\bar{W}'$ be 
    the set of all reduced words $w \in F_n$, also including $\epsilon$, such that 
    $w_1 \neq a_1$, $w_1w_2 \neq a_2a_1^{-1}$,
    $w_{fin} \neq a_1^{-1}$ and 
    $w_{fin-1}w_{fin} \neq a_1a_2^{-1}$.
    Now, we define $\bar{W} = \bar{W}' \cup \{a_1^{-1}\} \ \setminus \ \{a_2\}$.
    Then
    the set $\bar{B} = \{\widehat{p_w} \mid w \in \bar{W}\}$
    is a basis of $\widehat{C}(F_n)$. \\ \\
    3) Let $W_{Br}' = \bar{W} \cup \{a_1, a_2\} \ \setminus \ \{\epsilon\}$.
    Now, let $W_{Br}$ be a subset of $W_{Br}'$ such that for every
    $w \in W_{Br}'$
    it holds that $|\{w, w^{-1}\} \cap W_{Br}| = 1$ 
    (i. e. for every word in $W_{Br}'$ the set
    $W_{Br}$ either contains the word, or its inverse). 
    Then $B_{Br} = 
    \{\widehat{\phi_w} \mid w \in W_{Br}\}$ is a basis of $\widehat{Br}$.
\end{theorem}

\vspace{\baselineskip}

These two theorems have already been laid down and proven in \cite{ggd}
, with the difference being that $a_2$ was also included in the $\bar{W}$ set
for the case of $F_n$, which is corrected here, and that
the spanning set of $K_{Br}$ provided here is more precise (in the
original paper it was described as $\{ s_w, l_w, r_w \mid w \in F_n\}$).
Our main focus is attempting to provide
shorter and simpler proofs for these facts.

\subsection{Preliminaries}

\begin{lemma} \label{lemm:cbas}
    Elementary counting functions 
    for words of non-zero length form
    a basis of $C(M_n)$ (respectively $C(F_n)$).
\end{lemma}
\begin{proof}
    All elementary counting 
    functions span $C$ by definition, and 
    $p_{\epsilon} = \sum_{|w|=1} p_w$, 
    thus the aforementioned set also
    spans the space. Now, let us take a non-trivial
    linear combination of elementary counting functions for words of non-zero length, 
    $$f = \sum_{i=1}^{m}x_ip_{w_i},$$
    and let $i$ be the index of the
    shortest word $w_i$ in it. Then
    $f(w_i) = x_i$, thus $f \not\equiv 0$.
\end{proof}

\begin{corr}
    Let $W' = F_n \ \setminus \ \{\epsilon\}$. Now let $W$ be a subset of $W'$
    such that for every $w \in W'$ it holds that $|\{w, w^{-1}\} \cap W| = 1$.
    Then $\{\phi_w \mid w \in W\}$ is a basis of $Br(F_n)$.
\end{corr}

\begin{lemma} \label{lem:ext}
    For any word $w \in M_n$ (reduced word $w \in F_n$) the 
    left and right extension relations $l_w$, $r_w$ lie in $K$.
\end{lemma}
\begin{proof}
    Let us only provide the proof for $M_n$, since for $F_n$
    it is virtually identical.
    Any counting function is $0$ at $\epsilon$, thus for
    boundedness we only consider arguments of non-zero length.
    For words $w$ and $v$ with $|v| \geq 1$ 
    $$p_w(v) = [w = v_1 \ldots v_{|w|}] + p_w(v_2 \ldots v_{fin})$$
    (the number of subwords of $v$ equal to $w$ that do not start
    with the first letter of $v$ plus $1$ if $w$ 
    is a prefix of $v$).
    Every subword counted in $p_w(v_2 \ldots v_{fin})$ is preceded
    by some letter in $v$, thus 
    $$p_w(v_2 \ldots v_{fin}) = \sum_{s \in S}p_{sw}(v),$$
    therefore $l_w(v) = p_w(v) - p_w(v_2 \ldots v_{fin}) = [w = v_1 \ldots v_{|w|}]$
    may be equal to $1$ or $0$, thus $l_w$ is bounded.
    The boundedness of $r_w$
    may be proven similarly 
    by using the last letter of $v$ instead
    of the first.
\end{proof}

\begin{defn} \label{def:ind}
    We define $\delta_1(w, v) = [w = v_1 \ldots v_{|w|}]$,
    i. e. as $1$ if $v$ starts with $w$, and
    $0$ otherwise. Similarly, we define 
    $\delta_{fin}(w, v) = [w = v_{fin-|w|+1} \ldots v_{fin}]$,
    i. e. as $1$ if $v$ ends with $w$, and $0$ otherwise. 
    It is clear from the proof above that 
    $l_w(v) = \delta_1(w, v)$ and $r_w(v) = \delta_{fin}(w, v)$.
\end{defn}

\begin{lemma}
    For any reduced word $w \in F_n$ the 
    symmetry relation and the symmetrized extension relations
    $s_w$, $se_w$ lie in $K_{Br}$.
\end{lemma}
\begin{proof}
    First,
    $$\sigma(s_w) = \widehat{\phi_w} + \widehat{\phi_{w^{-1}}} = 
    \widehat{p_w} - \widehat{p_{w^{-1}}} + \widehat{p_{w^{-1}}} - \widehat{p_w} = \widehat{0},$$
    therefore $s_w \in K_{Br}$. Now,
    $$\sigma(se_w) = \widehat{\phi_w} -
    \sum_{s \in \bar{S} \setminus \{w_1^{-1}\}} \widehat{\phi_{sw}} -
    \Bigl(\widehat{\phi_{w^{-1}}} -
    \sum_{s \in \bar{S} \setminus \{w_1\}} \widehat{\phi_{w^{-1}s}}\Bigr) = $$$$
    (\widehat{\phi_w} - \widehat{\phi_{w^{-1}}}) - 
    \sum_{s \in \bar{S} \setminus \{w_1^{-1}\}} (\widehat{\phi_{sw}} -
    \widehat{\phi_{w^{-1}s^{-1}}}) =
    2\widehat{\phi_{w}} - 2\sum_{s \in \bar{S} \setminus \{w_1^{-1}\}} \widehat{\phi_{sw}} =$$$$
    2\Bigl(\widehat{p_w} - \sum_{s \in \bar{S} \setminus \{w_1^{-1}\}} \widehat{p_{sw}} - 
    \Bigl(\widehat{p_{w^{-1}}} - \sum_{s \in \bar{S} \setminus \{w_1^{-1}\}} 
    \widehat{p_{w^{-1}s^{-1}}}\Bigr)\Bigr) =
    2(\widehat{l_w} - \widehat{r_{w^{-1}}}),$$
    and it follows from Lemma \ref{lem:ext} that 
    $l_w$ and $r_w$ are bounded, thus
    $se_w \in K_{Br}$ too.
\end{proof}

\section{Spanning sets of $\widehat{C}$} \label{sec:bas}
In this section we prove that the sets $B$, $\bar{B}$ and $B_{Br}$ defined in
Theorem \ref{teor:bas} are spanning sets of $\widehat{C}(M_n)$, $\widehat{C}(F_n)$
and $\widehat{Br}(F_n)$ respectively. In the next section we 
will prove that the aforementioned sets are linearly independent.
In this section we write
$p_w$ instead of $\widehat{p_w}$, since we will be
only working with equivalency classes of functions. 

\subsection{Monoid case} \label{subs:mn}
\begin{lemma} \label{lemm:mspan}
    $B$ is a spanning set of $\widehat{C}(M_n)$.
\end{lemma}

\begin{proof}
Let us prove that all the elementary counting functions of non-empty words lie in
the span of $B$. Any word $w \in M_n$ with $|w| \geq 0$ can be represented
as $w = a_1^kva_1^m$ with $|v| \geq 0$, $v_1 \neq a_1$, $v_{fin} \neq a_1$
and $k,m \geq 0$. Let us define $||w|| = k + m$. Note that 
if $w = a_1^d$ the choice of $k$ and $m$ is ambiguous, but
$||a_1^d|| = d$, thus the definition is correct. In
this special case we presume $k = 1$ and $m = ||w|| - 1$.
Now, we prove that for any word $w$ with $|w| \geq 1$ the function class
$p_w$ lies in the span of $B$ inductively by $||w||$.

\paragraph{}
For $||w|| = 0$ it holds that $w_1 \neq a_1$ and $w_{fin} \neq a_1$, thus
$p_w \in B$ by definition.

\paragraph{}
Now let us prove the induction step.
Consider it known that for any word $w$ with $||w|| \leq d$ the function $p_{w}$ belongs to the span of $B$.
Let $w' = a^k_1wa^m_1$ with $||w'|| = k + m = d + 1$.
If $k > 0$, then by the left extension relation
$$p_{w'} = p_{a^k_1wa^m_1}= 
p_{a^{k-1}_1wa^m_1} - \sum_{s \in S \setminus \{a_1\}}p_{sa^{k-1}_1wa^m_1}.$$
If $|w| > 0$, then
$||a^{k-1}_1wa^m_1|| = k - 1 + m = d$ and
 $||sa^{k-1}_1wa^m_1|| = m \leq d$ (since $s \neq a_1$).
Otherwise, $k = 1$, thus $||a^{k-1}_1wa^m_1|| = ||a^m_1|| = m = d$ and
$||sa^{k-1}_1wa^m_1|| = ||sa^m_1|| = m = d$.
Therefore all the terms on the left side lie in the span of $B$ by the 
induction hypothesis, and $p_{w'}$ also lies in the span.

\paragraph{}
Now if $k = 0$, then $|w| > 0$ and $m > 0$. In this case
 $w' = wa^m_1$ and by the right extension relation
 $$p_{w'} = p_{wa^m_1}= 
p_{wa^{m-1}_1} - \sum_{s \in S \setminus \{a_1\}}p_{wa^{m-1}_1s}.$$
In this case $||wa_1^{m-1}|| = m-1 = d$ and $||wa^{m-1}_1s|| = 0$, 
thus $p_{w'}$ lies in the span again.

\paragraph{}
Therefore for all the words $w \in M_n$ with $|w| \geq 1$ the
class $p_w$ lies in the span of $B$, thus by Lemma \ref{lemm:cbas} the 
whole space $\widehat{C}(M_n)$ is spanned by $B$.
\end{proof}

\subsection{Group case} \label{subs:fn}
\begin{lemma} \label{lemm:fspan}
    $\bar{B}$ is a spanning set of $\widehat{C}(F_n)$.
\end{lemma}
\begin{proof}
The idea of the proof is the same as in the monoid case. However, it
requires a little more legwork since there are more cases.

\paragraph{}
As a preliminary step, let us prove that $p_{a_2}$ lies in the span of $\bar{B}$.
\begin{prop}
The counting function
$$f = \sum_{s\in\bar{S} \setminus \{a_1, a_1^{-1}\}}p_s - 
\sum_{\substack{s_1\in\bar{S}\setminus\{a_1\}, \\
s_2\in\bar{S}\setminus\{a_1^{-1}, s_1^{-1}\}}} p_{s_1s_2}$$
lies in $K(F_n)$.
\end{prop}
\begin{proof}
Let us subtract $\sum_{s\in\bar{S} \setminus \{a_1, a_1^{-1}\}}l_s$ from $f$. 
In this expression chain all the letters $s, s_1, s_2$ are presumed to lie in $\bar{S}$
(thus all the sums are finite) and
all the words are presumed to be reduced so as
not to specify $s\in\bar{S}$, $s_1\in\bar{S}$ and $s_2\in\bar{S} \setminus \{s_1^{-1}\}$
in every sum. 
$$f - \sum_{s \neq a_1, a_1^{-1}}l_s = \sum_{s \neq a_1, a_1^{-1}}p_s
- \sum_{s \neq a_1, a_1^{-1}}l_s -
\sum_{s_1 \neq a_1, s_2\neq a_1^{-1}}p_{s_1s_2}
= $$$$
\sum_{s \neq a_1, a_1^{-1}}p_s - \sum_{s \neq a_1, a_1^{-1}}p_s + 
\sum_{s_2\neq a_1, a_1^{-1}}p_{s_1s_2} -
\sum_{s_1 \neq a_1, s_2\neq a_1^{-1}}p_{s_1s_2} = $$$$
\sum_{s\neq a_1,a_1^{-1}}p_{a_1s} - \sum_{s\neq a_1,a_1^{-1}}p_{sa_1} =
\sum_{s\neq a_1,a_1^{-1}}p_{a_1s} - \sum_{s\neq a_1,a_1^{-1}}p_{sa_1} + p_{a_1a_1} -
p_{a_1a_1} =$$$$ \sum_{s\neq a_1^{-1}}p_{a_1s} - \sum_{s\neq a_1^{-1}}p_{sa_1} =
l_{a_1} - r_{a_1} \sim 0.$$
Therefore $f \sim 0$.
\end{proof} 
All the terms in $f$ except $p_{a_2}$ lie in $\bar{B}$, therefore
$p_{a_2} \sim p_{a_2} - f \in $ span $\bar{B}$.

\paragraph{}
Now, let us introduce a few definitions. Let
$P=\{a_1,\ a_2a_1^{-1}\}$, and $P'=\{a_1^{-1},\
a_1a_2^{-1}\}$. Let us introduce the set
$$F_P = \{ a_1^k(a_2a_1^{-1})^m
\mid k, m \geq 0 \}.$$ Note, that words from $F_P$ are 
reduced as words of $F_n$, meaning that $a_1$ and
$a_1^{-1}$ do not stand next to each other and that
all the other words over the alphabet $P$ are
not reduced in this sense.
Similarly, let $$F_{P'} =
\{ (a_1a_2^{-1})^{m'}a_1^{-k'} \mid k', m' \geq 0 \}.$$
It is also clear that this is precisely the subset
of all the words over the alphabet $P'$ that are
reduced as words of $F_n$.

\paragraph{}
Then, for any word $w = a_1^k(a_2a_1^{-1})^m \in
F_P$ we define $|w|_P$ as $k + m$, and for any
$w' = (a_1a_2^{-1})^{m'}a_1^{-k'} \in F_{P'}$ 
we define $|w'|_{P'}$ as $k' + m'$. 

\paragraph{}
Now,
let us extend the definitions of $|.|_P$ and
$|.|_{P'}$ to all words in $F_n$. 
The word $w$ can be represented in two forms: 
$$w = a_1^k(a_2a_1^{-1})^mv$$ and 
$$w = v'(a_1a_2^{-1})^{m'}a_1^{-k'}$$ with $k, m, k', m' 
\geq 0$. Let us define $|w|_P$ as $max(k + m)$ over
all representations of the first form, and 
$|w|_{P'}$ as $max(k' + m')$ over all 
representations of the second form.
Finally, let us introduce a metric called
defect of a word, defined as follows:
$$||w|| = |w|_P + |w|_{P'} = k + m + k' + m'.$$
Further, we will refer to these values
$k, m, k', m'$ of a word $w$ as 
$k_w, m_w, k'_w, m'_w$.

\paragraph{}
Less formally, this metric can be described as the
number of consecutive patterns from $P$ at the
beginning of the word plus the number of consecutive
patterns from $P'$ at the end of the word. 

\paragraph{}
For
example, $||a^2_1a_2a^{-1}_1a_3a_1a^{-1}_2|| = 4$. 
For brevity, we extend this notation to the 
elementary counting functions, so that 
$||p_{w}|| = ||w||$.

\paragraph{}
It is important to note that it does not always 
hold that $$w = 
a_1^{k_w}(a_2a_1^{-1})^{m_w}v
(a_1a_2^{-1})^{m'_w}a_1^{-k'_w}$$
for some word $v$. The cases where this does not
hold are when a letter $a_1$ participates in both
$a_1^{k_w}$ and $(a_1a_2^{-1})^{m'_w}$:
$$w = a_1^{k_w - 1}(a_1a_2^{-1})^{m'_w}a_1^{-k'_w},$$
and when $a_1^{-1}$ participates in
$(a_2a_1^{-1})^{m_w}$ and $a_1^{-k'_w}$:
$$w = a_1^{k_w}(a_2a_1^{-1})^{m_w}a_1^{-k'_w+1}.$$

\paragraph{}
It is obvious that words of $0$ defect except $a_2$ lie in $\bar{W}$. However, we
have already proven that $p_{a_2}$ lies in the span, thus
all the counting functions for words of $0$ defect lie in the span of $\bar{B}$.
On the other hand,
words of positive defect do not belong to $\bar{W}$, with the only
exception of $a_1^{-1}$. Now we prove
that any elementary counting function $p_w$ with
$||w|| > 0$ and $w \neq a_1^{-1}$ can be reduced to
a linear combination of
elementary counting functions with lesser defects by
right and left extension relations. If we
prove that, then by induction any elementary 
counting function will belong to $\bar{B}$.

\paragraph{}
First, let us consider the case of words of form
$w = a_1v$, meaning $k_w > 0$. 
Take an elementary counting function
$p_{a_1v}$. For this function the left
extension relation can be written as
\begin{equation} \label{eq:ler}
p_{a_1v} = p_{v} -\sum_{s\in\bar{S} \setminus \{a_1,v_1^{-1}\}}p_{sv}.
\end{equation}
First, we prove that $||p_{v}|| < ||p_w|| $. It is clear
that $k_v < k_w$, because $v$ is obtained from $w$ 
by removing the first letter $a_1$ and $m_v = m_w,
\ m'_v \leq m'_w$ and $k'_v = k'_w$,
thus $||p_v|| < ||p_w||$.

\paragraph{}
Now, we prove that $||p_{sv}|| < ||p_w||$ for every 
term
$p_{sv}$ in (\ref{eq:ler}). In order to do that,
we have to calculate
$||sv||$, where $s \neq a_1, v_1^{-1}$. The latter means
that $k_{sv} = 0 < k_w$. Furthermore, 
$m_{sv} = 0 \leq m_w$
even if $s = a_2$, because $v_1 \neq a_1^{-1}$. 
Then, $m'_{sv} \leq m'_w$, because $s \neq a_1$,
so it can not form a new $a_1a_2^{-1}$. Finally,
if $|v| > 0$, we have
$k'_{sv} \leq k'_w$, because $v_1 \neq a_1^{-1}$. 
In the case of $|v| = 0$, i. e. $v = \epsilon$, and $s = a_1^{-1}$ we get
$p_{sv} = p_s = p_{a_1^{-1}} \in \bar{B}$ by definition,
thus this term already lies in the basis.
This means that
$$||p_{sv}|| = k_{sv} + m_{sv} + k'_{sv} + m'_{sv} <
k_w + m_w + k'_w + m'_w = ||p_{w}||,$$
which concludes the proof of the first case.

\paragraph{}
Thus, we have established the proof of the induction
step for elementary counting functions of words with
$k_w > 0$. For counting functions of words
with $k'_w > 0$ ($w = va_1^{-1})$
this can be done symmetrically 
using the right extension relations, with the only 
difference being that if $|v'| = 0$ we already
have $p_w = p_{a_1^{-1}} \in \bar{B}$ and no additional proof
is required.

\paragraph{}
Now, we have to prove the induction step for the words of form  
$w = a_2a_1^{-1}v$ ($k_w = 0, \ m_w > 0$). We also assume that
$|v| > 0$ and $v_{fin} \neq a_1^{-1}$, i. e. that $k'_w = 0$.
Again, by left extension relation
\begin{equation}
 p_{a_2a_1^{-1}v} = p_{a_1^{-1}v} - \sum_{s\in\bar{S} \setminus \{a_2, a_1\}}p_{sa_1^{-1}v},
\end{equation}
and we again have to prove that all the terms on the right
side are either of lesser defect than $n$, or can be 
explicitly reduced to a linear combination of terms with 
lesser defect. 

\paragraph{}
First, let us find the upper bound for 
$||p_{a_1^{-1}v}||$. Let us denote $y = a_1^{-1}v$.
Obviously, $k_y = 0 = k_w$ and $m_y = 0 < m_w$. 
It is also clear that $k'_y \leq k'_w$ and 
$m'_y \leq m'_w$, since $y$ is obtained from $w$
by removing the first letter, thus 
$||p_y|| < ||p_w||$.

\paragraph{}
Now, to the other terms, which have the form 
$sa_1^{-1}v = sy$ with $s \neq a_1, a_2$. 
Due to the restrictions on $s$ we have 
$k_{sy} = m_{sy} = 0$, thus $k_{sy} = k_{w}$ and
$m_{sy} < m_{w}$. It also holds that $m'_{sy} = m'_w$ since $s \neq a_1$
and thus no new $a_1a_2^{-1}$ pattern could be formed. Finally,
$k'_{sy} = 0 = k'_w$, therefore $||p_{sy}|| < ||p_w||$.

\paragraph{}
In the case of $k_w = m_w = k_w' = 0$ and $m_w' > 0$ we can again prove
the induction step symmetrically using right extension relations.

\paragraph{}
Thus we have completed the proof of the induction
step for all word forms, and therefore
in the group case $\bar{B}$ also spans 
$\widehat{C}(F_n)$.
    
\end{proof}

\subsection{Brooks space}
\begin{lemma} \label{lemm:bspan}
    $B_{Br}$ is a spanning set of $\widehat{B}$.
\end{lemma}

\begin{proof}
    Let $\phi_w = p_w - p_{w^{-1}}$ be the Brooks quasimorphism for some word
    $w$. It follows from Lemma \ref{lemm:fspan} that 
    $p_w = \sum_{i=0}^m x_{i} p_{w_i}$ with $p_{w_i} \in \bar{B}$.
    It is clear that $p_{w^{-1}} = \sum_{i=0}^m x_{i} p_{w_i^{-1}}$, 
    thus $\phi_w =  \sum_{i=0}^m x_{i} \phi_{w_i}$.
    Let us presume that $w_i \neq \epsilon$ for every $i$, since 
    $\phi_{\epsilon} \equiv 0$.
    Now for every
    $i$ either $w_i$ or $w_i^{-1}$ lies in $W_{Br}$. In the first case, 
    $\phi_{w_i} \in B_{Br}$. In the second case, $x_i \phi_{w_i} = 
    -x_i \phi_{w_i^{-1}}$, and $\phi_{w_i^{-1}} \in B_{Br}$.
    Therefore if we define $I_w = \{i \mid 0 \leq i \leq m, w_i \in W_{Br}\}$
    and $I_w^{-1} = \{i \mid 0 \leq i \leq m, w_i^{-1} \in W_{Br}\}$, then
    $$\phi_w = \sum_{i\in I_w} x_{i} \phi_{w_i} - 
    \sum_{i\in I_w^{-1}} x_{i} \phi_{w_i^{-1}}$$
    lies in the span,
    and $B_{Br}$ is indeed a spanning set of $\widehat{B}$.
\end{proof}

\section{Linear independence of $B$, $\bar{B}$ and $B_{Br}$}
\subsection{Monoid case}
\begin{lemma} \label{lemm:mind}
    The set $B$ is linearly independent.
\end{lemma}
\begin{proof}
    Let $f = \sum_{i=0}^m x_i p_{w_i} \in C(M_n)$ be a linear combination
    of elementary counting functions 
    with $w_i \in W$ and $x_i \neq 0$. Let us prove that 
    $\widehat{f} \neq \widehat{0}$. For a word $u$ let $x_u$ be the coefficient at
    this word in $f$ ($x_u$ may be zero).
    Let $w$ be one of the the shortest words $w_i$, and let $L$ be the depth of $f$.

    \paragraph{}
    First, we consider the case when $w = \epsilon$. In this case
    $f(a_1^k) = kx_{\epsilon}$, thus $f$ is not bounded.

    \paragraph{}
    Now, if $|w| \geq 1$, let $v = wa_1^L$. Let us consider the word
    $v^k = wa_1^Lwa_1^L\ldots wa_1^L$ for some $k > 0$. Every subword
    of $v^k$ of length $l$ with $|w| \leq l \leq L$ is either equal to $w$,
    or starts or ends with $a_1$. Therefore $f(v^k) = kx_w \neq 0$, 
    and $f$ is not a bounded function.

    \paragraph{}
    Now, since every non-trivial linear combination of functions in
    $B$ is not equal to $\widehat{0}$, it follows that $B$ is linearly independent.
\end{proof}

\subsection{Group case}
\begin{lemma} \label{lemm:find}
    The set $\bar{B}$ is linearly independent.
\end{lemma}
\begin{proof}
    Again, let $f = \sum_{i=0}^m x_i p_{w_i} \in C(F_n)$ be a linear combination
    of elementary counting functions 
    with $w_i \in \bar{W}$ and $x_i \neq 0$. Let us prove that $f$ is
    not bounded. We define $x_u$ and $L$ similarly to the 
    free monoid case.

    \paragraph{}
    First, let $x_{\epsilon} \neq 0$. In this case, again, 
    $f(a_1^k) = kx_{\epsilon}$ and $f$ is not bounded.

    \paragraph{}
    Second, let $x_{\epsilon} = 0$ and $x_{a_1^{-1}} \neq 0$. Then
    $f(a_1^{-k}) = kx_{a_1^{-1}}$ and $f$ is not bounded.

    \paragraph{}
    Now, let $x_{\epsilon} = 0$, $x_{a_1^{-1}} = 0$ and presume that 
    $f$ is bounded.

    \begin{defn}
        A reduced word $w$ with $|w| \geq 1$ is considered clean if
        $w_1 \neq a_1, a_1^{-1}$ and
        $w_{fin} \neq a_1, a_1^{-1}$.
    \end{defn}

    \begin{prop} \label{prop:1}
        Let $w$ be a clean word and let $v = wa_1^L$. Then $f(v^k) = kf(v)$.
        Similarly, if $v' = a_1^{-L}w$, then $f(v'^k) = kf(v')$
    \end{prop}
    \begin{proof}
        $v^k = wa_1^Lwa_1^L\ldots wa_1^L$. Every subword
        of $v^k$ of length $l \leq L$ 
        either entirely lies in a copy of $v = wa_1^L$, or
        starts with $a_1$. Since $x_{a_1^ku} = 0$ for any $u$ and $k > 0$,
        it indeed holds that $f(v^k) = kf(v)$. The proof for $v'$
        is virtually identical, since it is also true that $x_{ua_1^{-k}} = 0$
         for any $u$ and $k > 0$.
    \end{proof}

    Therefore if $f$ is bounded, then $f(wa_1^L) = f(a_1^{-L}w) = 0$ for
    every clean word $w$.

    \paragraph{}
    Now, suppose that there exists a clean word $w$ such that
    $f(a_1^{-L}wa_1^L) = c \neq 0$. Then let us take the word
    $$v = wa_1^La_2a_1^{-L}.$$
    For this word $v^k = wa_1^La_2a_1^{-L}wa_1^La_2a_1^{-L} \ldots wa_1^La_2a_1^{-L}$.
    Since $x_{a_1^k} = 0$ for any $k \in \mathbb{Z}$ and $x_{a_2} = 0$, it is clear that 
    $$f(v^{k}) = f(wa_1^L) + f(a_1^La_2a_1^{-L}) + \sum_{i=1}^{k-1}
    (f(a_1^{-L}wa_1^L) + f(a_1^La_2a_1^{-L})) = $$$$
    f(wa_1^L) + f(a_2) + \sum_{i=1}^{k-1}
    (f(a_1^{-L}wa_1^L) + f(a_2)) = 0+0+\sum_{i=1}^{k-1}
    (c + 0) = c(k-1).$$
    Therefore if $f$ is bounded, then $f(a_1^{-L}wa_1^L) = 0$ for every clean word $w$.

    \paragraph{}
    Now, let $f$ still be bounded, and suppose there exists
    a clean word $w$ such that $f(w) \neq 0$. In this case, if we take the word
    $$v = wa_1^Lwa_1^{-L},$$
    then for this word $v^k = wa_1^Lwa_1^{-L}wa_1^Lwa_1^{-L}\ldots wa_1^Lwa_1^{-L}$,
    and similarly to the previous case
    $$f(v^k) = f(wa_1^{L}) + f(w) + \sum_{i=1}^{k-1}(f(a_1^{-L}wa_1^L) + f(w)) =
    0 + f(w) + \sum_{i=1}^{k-1}(0 + f(w)) = kf(w),$$
    thus $f$ is still unbounded. Therefore since we presume $f$ to be bounded,
    it also holds that
    $f(w) = 0$ for any clean word $w$. In particular, it follows that
    $x_{s} = 0$ for any $s \in \bar{S}^1$.

    \paragraph{}
    Now, let $w$ be a word of smallest length such that $x_w \neq 0$.
    Since we have presumed $f$ to be bounded, $|w| > 1$.
    Now $f(w) = x_w \neq 0$, thus $w_1 = a_1^{-1}$ or
    $w_{fin} = a_1$, or both. 
    If $w_1 = a_1^{-1}$ and $w_{fin} \neq a_1$, then let $w' = a_2w$.
    If $w_1 \neq a_1^{-1}$ and $w_{fin} = a_1$, then let $w' = wa_2^{-1}$.
    Finally, if $w_1 = a_1^{-1}$ and $w_{fin} = a_1$, then let $w' = a_2wa_2^{-1}$.

    \paragraph{}
    Now, let 
    $$\Delta_1^f(w) = \sum_{\substack{v \in F_n, 
    2 \leq |v| \leq |w|}}x_{v}\delta_1(v, w),$$
    and let
    $$\Delta_{fin}^f(w) = \sum_{\substack{v \in F_n, 
    2 \leq |v| \leq |w|}}x_{v}\delta_{fin}(v, w)$$
    with $\delta_1$, $\delta_{fin}$ defined in \ref{def:ind}.
    These functions represent restrictions of $f$ to counting subwords of length at
    least $2$ that $w$ starts with (respectively, ends with). Now note that 
    if $w_1 = a_1^{-1}$, then $w'$ starts with $a_2a_1^{-1}$ and thus $\Delta_1^f(w') = 0$.
    Similarly, if $w_{fin} = a_1$, it holds that $\Delta_{fin}^f(w') = 0$

    \paragraph{}
    In the first case of $w'$ the function value $f(w')$ can be represented as 
     $$f(w') = f(a_2w) = f(w) + f(a_2) + \Delta_1^f(w') = f(w) + 0 + 0 = f(w) \neq 0.$$
    In the second case 
    $$f(w') = f(wa_2^{-1}) = f(w) + f(a_2^{-1}) + \Delta_{fin}^f(w') = f(w) + 0 + 0 = f(w) \neq 0.$$
    Finally, in the third case
    $$f(w') = f(a_2wa_2^{-1}) = f(w) + (f(a_2) + \Delta_1^f(w')) + (f(a_2^{-1}) + \Delta_{fin}^f(w')) - x_{w'} = $$$$
    f(w) + 0 + 0 - 0 = f(w) \neq 0,$$
    with the last term included due to the fact that $w'$ itself is counted in both
    $\Delta_1^f(w')$ and $\Delta_{fin}^f(w')$.
    In all three equations the second
    step follows from the fact that $f(a_2) = f(a_2^{-1}) = 0$ and
    that $x_v = 0$ for any word $v$ that starts with $a_2a_1^{-1}$ or ends with $a_1a_2^{-1}$.
    However, in all three cases $w'$ is a clean word, thus $f(w')$ must be equal to $0$.
    Thus we have arrived at a contradiction and $f$ can not be bounded.
\end{proof}

\subsection{Brooks space}
\begin{lemma} \label{lemm:bind}
    The set $B_{Br}$ is linearly independent.
\end{lemma}
\begin{proof}
    Let $f = \sum_{i=0}^m x_i \phi_{w_i} = 
    \sum_{i=0}^m x_i p_{w_i} - \sum_{i=0}^m x_i p_{w_i^{-1}}$
    be some non-trivial linear combination of
    quasimorphisms for words $w_i \in W_{Br}$. It holds that if $w_i \neq a_1, a_1^{-1},
    a_2, a_2^{-1}$, then both $p_{w_i}$ and $p_{w_i^{-1}}$ lie in $\bar{B}$.
    Therefore $f = x\phi_{a_1} + x'\phi_{a_2} + f'$, where $f'$ is a linear combination
    of elements of $\bar{B}$. If $x = x' = 0$, then $f'$ must be non-trivial, thus
    $f$ is not bounded by Lemma \ref{lemm:find}. If $x \neq 0$, then
    $f(a_1^k) = xk$, and thus $f$ is again not bounded.

    \paragraph{}
    Now, let $x = 0$ and $x' \neq 0$,
    i. e. $f = x'\phi_{a_2} + f'$, where $f'$ is either trivial, or not bounded. 
    Similarly to the group case, let
    us presume $f$ to be bounded and take the word
    $$v = a_2a_1^La_2a_1^{-L}.$$
    Similarly to the group case proof,
    $$f(v^k) = f(a_2a_1^L) + f(a_2) +
    (k-1)(f(a_1^{-L}a_2a_1^L) + f(a_2)),$$ therefore
    $f(a_1^{-L}a_2a_1^L) = -f(a_2) = -x' \neq 0$. On the other hand, $f(a_2) = -f(a_2^{-1})$,
    since $f$ consists of quasimorphisms, 
    therefore $f(a_1^{-L}a_2a_1^L) = f(a_2^{-1})$.
    Now let 
    $$v' = a_2a_1^La_2^{-1}a_1^{-L}.$$
    Again, 
    $$f(v'^k) = f(a_2a_1^L) + f(a_2^{-1}) + (k-1)(f(a_1^{-L}a_2a_1^L) + f(a_2^{-1})) =$$$$
    f(a_2a_1^L) + f(a_2^{-1}) + 2(k-1)f(a_2^{-1}),$$
    and since $f(a_2^{-1}) \neq 0$, it follows that $f$ can not be bounded.
    \end{proof}

\section{Proof conclusions}
\begin{proof}[Proof (Theorem \ref{teor:bas})]
    From Lemma \ref{lemm:mspan} it follows that the set $B$ spans $\widehat{C}(M_n)$,
    and from Lemma \ref{lemm:mind} it follows that $B$ is linearly independent,
    thus it is a basis of $\widehat{C}(M_n)$. Similarly, from Lemma 
    \ref{lemm:fspan} it follows that the set $\bar{B}$ spans $\widehat{C}(F_n)$,
    and from Lemma \ref{lemm:find} it follows that $\bar{B}$ is linearly independent,
    thus it is a basis of $\widehat{C}(F_n)$. Finally,
    from Lemma 
    \ref{lemm:bspan} it follows that the set $B_{Br}$ spans $\widehat{Br}(F_n)$,
    and from Lemma \ref{lemm:bind} it follows that $B_{Br}$ is linearly independent,
    thus it is a basis of $\widehat{Br}(F_n)$
\end{proof}

The proof of Theorem \ref{teor:ker} can be easily derived
from the proofs given in Section \ref{sec:bas}.

\begin{proof}[Proof (Theorem \ref{teor:ker})]
    In Section \ref{sec:bas} we have proven that any
    elementary counting function (and therefore any counting function) may be 
    represented as
    $$\sum_{i=0}^m x_i p_{w_i} + \sum_{j=0}^{m'} y_j l_{v_j} +
    \sum_{k=0}^{m''} z_k r_{u_k}$$
    with $p_{w_i} \in B$ (respectively $\bar{B}$) being basis elements, and $l_{v_j}$ and 
    $r_{u_k}$ being left and right extension relations for some
    words. Therefore any two counting functions that lie in
    the same equivalency class with respect to the relation $\sim$ differ by
    a linear combination of left and right extension relations, 
    thus $l_w$ and $r_w$ span $K$, which concludes the proof for the monoid
    and group cases.

    \paragraph{}
    Finally, for the Brooks space let us prove that for $\sigma_1$ and
    $\sigma_2$ from (\ref{eq:kern}) it stands that the functions $s_w$ span $\Ker \sigma_1$
    and the functions $se_w$ span $\sigma_1^{-1}(\Ker \sigma_2)$.
    Since $\sigma_1(p_w) = p_w - p_w^{-1}$, it is clear that $\Ker \sigma_1$ is
    exactly the set of all symmetric counting functions, thus it is
    spanned by $p_w + p_{w^{-1}} = s_w$.
    Next, from the same proofs in Section \ref{sec:bas} it follows that any linear combination
    of Brooks quasimorphisms can be represented in the following form:
    $$\sum_{i=0}^{m'''} x_i' \phi_{w'_i} = \sum_{i=0}^{m'''} x_i' p_{w'_i}
    -\sum_{i=0}^{m'''} x_i' p_{w_i'^{-1}} = $$$$
    \Bigl(\sum_{i=0}^m x_i p_{w_i} + \sum_{j=0}^{m'} y_j l_{v_j} +
    \sum_{k=0}^{m''} z_k r_{u_k}\Bigr) - 
    \Bigl(\sum_{i=0}^m x_i p_{w_i^{-1}} + \sum_{j=0}^{m'} y_j r_{v_j^{-1}} +
    \sum_{k=0}^{m''} z_k l_{u_k^{-1}}\Bigr) = $$$$
    \Bigl(\sum_{i=0}^m x_i p_{w_i} - \sum_{i=0}^m x_i p_{w_i^{-1}}\Bigr) +
    \sum_{j=0}^{m'} y_j (l_{v_j} - r_{v_j^{-1}}) - 
    \sum_{k=0}^{m''} z_k (l_{u_k^{-1}} - r_{u_k}) =$$$$
    \sum_{i=0}^m x_i \phi_{w_i} + \sum_{j=0}^{m'} y_j se_{v_j} - 
    \sum_{k=0}^{m''} z_k se_{u_k^{-1}}$$
    with $\phi_{w_i} \in B_{Br}$ also being basis elements,
    therefore the pre-image of $\Ker \sigma_2$ is 
    spanned by $se_w$, and from (\ref{eq:kern}) it follows that
    $K_{Br}$ is indeed spanned by $s_w$ and $se_w$.
\end{proof}

\bigskip

\noindent\textbf{Author's address:}\\
\noindent Petr Kiyashko \\
\noindent \textsc{HSE University, Faculty of Mathematics,\\
6 Usacheva str., 119048 Moscow, Russia}\\
\texttt{pskiyashko@edu.hse.ru};

\end{document}